\documentclass[12pt]{article}
\usepackage{amsmath,enumerate,amsfonts,amssymb,color,graphicx,amsthm}
\usepackage{enumitem}

\usepackage{multirow,caption,subcaption}
\usepackage{hyperref}
\usepackage{todonotes}
\usepackage[normalem]{ulem}
\usepackage{cite}

\numberwithin{equation}{section}

\def\ep{{\epsilon}}

\def\p{{\partial}}

\newcounter{marnote}

\begin{document}
\newtheorem{lem}{Lemma}[section]
\newtheorem{rem}{Remark}
\newtheorem{question}{Question}
\newtheorem{prop}{Proposition}
\newtheorem{cor}{Corollary}
\newtheorem{thm}{Theorem}[section]
\newtheorem{definition}{Definition}[section]
\newtheorem{openproblem}{Open Problem}
\newtheorem{conjecture}{Conjecture}

\newenvironment{dedication}
  {
   \vspace*{\stretch{.2}}
   \itshape             
  }
  {
   \vspace{\stretch{1}} 
  }

\title
{Existence results for Toda systems with sign-changing prescribed functions: Part II
\thanks{ sunll\underline{~}m@163.com (Linlin Sun); zhuxiaobao@ruc.edu.cn (Xiaobao Zhu)}}
\author{
Linlin Sun\\
School of Mathematics and Computational Science\\
Xiangtan University\\
Xiangtan 411105, P. R. China\\
\\
Xiaobao Zhu\\
School of Mathematics\\
Renmin University of China\\
Beijing 100872, P. R. China\\
}

\date{ }
\maketitle

\begin{abstract}

Let $(M, g)$ be a compact Riemann surface with area $1$. We investigate the Toda system
\begin{align}
\label{eq-jlw}
\begin{cases}
-\Delta u_1 = 2\rho_1(h_1e^{u_1}-1) - \rho_2(h_2e^{u_2}-1),\\
-\Delta u_2 = 2\rho_2(h_2e^{u_2}-1) - \rho_1(h_1e^{u_1}-1),
\end{cases}
\end{align}
on $(M, g)$ where $\rho_1, \rho_2 \in (0,4\pi]$, and $h_1$ and $h_2$ are two smooth functions on $M$.
When some $\rho_i$ equals $4\pi$, Eq. \eqref{eq-jlw} becomes critical with respect to the Moser-Trudinger inequality for the Toda system, making the existence problem significantly more challenging. In their seminal article (Comm. Pure Appl. Math., 59 (2006), no. 4, 526--558), Jost, Lin, and Wang established sufficient conditions for the existence of solutions to Eq. \eqref{eq-jlw} when $\rho_1=4\pi$, $\rho_2 \in (0,4\pi)$ or $\rho_1=\rho_2=4\pi$, assuming that $h_1$ and $h_2$ are both positive. In our previous paper we extended these results to allow $h_1$ and $h_2$ to change signs in the case $\rho_1=4\pi$, $\rho_2 \in (0,4\pi)$.
In this paper we further extend the study to prove that Jost-Lin-Wang's sufficient conditions remain valid even when $h_1$ and $h_2$ can change signs and $\rho_1=\rho_2=4\pi$. Our proof relies on an improved version of the Moser-Trudinger inequality for the Toda system, along with dedicated analyses similar to Brezis-Merle type and the use of Pohozaev identities.
\end{abstract}

\setcounter {section} {0}

\section{Introduction}
Let $(M, g)$ be a compact Riemann surface with area $1$ , and let $h(x)$  be a smooth function on $M$ . The celebrated Kazdan-Warner problem \cite{KW74} seeks to understand under what conditions on the prescribed function $h$  the following sub-linear elliptic partial differential equation has a solution:
\begin{align}\label{eq-kw}
-\Delta u = 8\pi (he^u - 1).
\end{align}
This problem is often referred to as the ``Nirenberg problem" when $M$  is the standard sphere, and it has been extensively studied \cite{M, KW74, CD87, CY87, CY88, H90, CL93, XY93, CL95, St05, CLLX}, among others. When $M$  is a general Riemann surface, Eq. \eqref{eq-kw} arises in the context of the so-called Chern-Simons Higgs theory \cite{HKP, JWE, CLP, Tar96}, among others. The coefficient $8\pi$  in Eq. \eqref{eq-kw} is critical with respect to the Moser-Trudinger inequality (cf. \cite{Fon, DJLW97}):
\begin{align}
\log \int_M e^u \leq \frac{1}{16\pi} \int_M |\nabla u|^2 + \int_M u + C.
\end{align}
Thus, the existence problem for Eq. \eqref{eq-kw} becomes intricate. Ding-Jost-Li-Wang \cite{DJLW97} addressed this problem using a variational approach by minimizing the functional
\begin{align}
I(u) = \frac{1}{2} \int_M |\nabla u|^2 + 8\pi \int_M u
\end{align}
in the Hilbert space
\begin{align}
X = \left\{ u \in H^1(M) : \int_M h e^u = 1 \right\}.
\end{align}
Assuming $h$  is positive, they showed that if
\begin{align}\label{cond-djlw}
\Delta \log h(p) + 8\pi - 2K(p) > 0,
\end{align}
where $K$  is the Gaussian curvature of $M$ , and $p$  is any maximum point of the sum of $2\log h$ and the regular part of the Green function, then $I$  attains its infimum in $X$  and Eq. \eqref{eq-kw} has a minimal solution. Yang and the second author \cite{YZ17} relaxed the positivity condition on $h$  to nonnegativity. Recently, the first author and Zhu \cite{SZ24} and the second author \cite{Z24} independently proved that the condition \eqref{cond-djlw} remains sufficient even for sign-changed prescribed functions. All these works are based on the variational approach. These results were also obtained using the flow method \cite{LZ19, SZ21, WY22, LX22}.

In this paper, we continue to investigate the Toda system \eqref{eq-jlw}, which can be viewed as the Frenet frame of holomorphic curves in $\mathbb{CP}^2$ (see \cite{G97}) from a geometric perspective, and also arises in physics in the study of the nonabelian Chern-Simons theory in the self-dual case, where a scalar Higgs field is coupled to a gauge potential; see \cite{Dun95,Ya01,Tar08}. Our focus is on the existence result, and we aim to explore the variational approach developed in \cite{DJLW97,JW01,LL05,JLW06}. Recall that \eqref{eq-jlw} represents the critical point of the functional
\begin{align*}
J_{\rho_1,\rho_2}(u_1,u_2) = \frac{1}{3}\int_M \left(|\nabla u_1|^2+\nabla u_1\nabla u_2+|\nabla u_2|^2\right)
+ \rho_1\int_M u_1 + \rho_2\int_M u_2
\end{align*}
in the Hilbert space
\begin{align*}
\mathcal{H} = \left\{(u_1,u_2)\in H^1(M)\times H^1(M):~\int_M h_1e^{u_1}=\int_M h_2e^{u_2}=1\right\}.
\end{align*}
From the Moser-Trudinger inequality for the Toda system
\begin{align}\label{ineq-toda}
\inf_{(u_1,u_2)\in\mathcal{H}} J_{\rho_1,\rho_2} \geq -C~~~~\text{iff}~~~~\rho_1,\rho_2\in(0,4\pi],
\end{align}
derived by Jost-Wang \cite{JW01}, it is known that $J_{\rho_1,\rho_2}$ is coercive and attains its infimum when $\rho_1,\rho_2\in(0,4\pi)$. However, when either $\rho_1$ or $\rho_2$ equals $4\pi$, the existence problem becomes more intricate. In this paper, we shall focus on minimal type solutions. Consequently, we assume $\rho_i \leq 4\pi$, $i=1,2$, throughout the discussion.

Let us review the existence result when one of $\rho_1$ and $\rho_2$ equals $4\pi$, which was obtained by Jost, Lin, and Wang when $h_1$ and $h_2$ are both positive.

\begin{thm}[Jost-Lin-Wang \cite{JLW06}]
    \label{thm-jlw1}
    Let $(M,g)$ be a compact Riemann surface with Gaussian curvature $K$. Let $h_1,h_2 \in C^2(M)$ be two positive functions and $\rho_2 \in (0, 4\pi)$. Suppose that
    \begin{align}\label{JLW-cond-1}
    \Delta \log h_1(x) + (8\pi - \rho_2) - 2K(x) > 0,~~~\forall x \in M,
    \end{align}
    then $J_{4\pi,\rho_2}$ has a minimizer $(u_1, u_2) \in \mathcal{H}$ which satisfies
    \begin{align}\label{eq-jlw-1}
    \begin{cases}
    -\Delta u_1 = 8\pi(h_1 e^{u_1} - 1) - \rho_2(h_2 e^{u_2} - 1),\\
    -\Delta u_2 = 2\rho_2(h_2 e^{u_2} - 1) - 4\pi(h_1 e^{u_1} - 1).
    \end{cases}
    \end{align}
\end{thm}

When $\rho_1 = \rho_2 = 4\pi$ and both $h_1$ and $h_2$ are positive, we have:

\begin{thm}[Li-Li \cite{LL05}, Jost-Lin-Wang \cite{JLW06}]
    \label{thm-critical-0}
    Let $(M, g)$ be a compact Riemann surface with Gaussian curvature $K$. Let $h_1, h_2 \in C^2(M)$ be two positive functions. Suppose that
    \begin{align}\label{JLW-cond-2}
    \min\{\Delta \log h_1(x), \Delta \log h_2(x)\} + 4\pi - 2K(x) > 0,~~~\forall x \in M,
    \end{align}
    then $J_{4\pi, 4\pi}$ has a minimizer $(u_1, u_2) \in \mathcal{H}$ which satisfies
    \begin{align}\label{eq-jlw-2}
    \begin{cases}
    -\Delta u_1 = 8\pi(h_1 e^{u_1} - 1) - 4\pi(h_2 e^{u_2} - 1),\\
    -\Delta u_2 = 8\pi(h_2 e^{u_2} - 1) - 4\pi(h_1 e^{u_1} - 1).
    \end{cases}
    \end{align}
\end{thm}

We remark that Li-Li obtained Theorem \ref{thm-critical-0} when $h_1 = h_2 = 1$ and Jost-Lin-Wang obtained it for general positive $h_1$ and $h_2$.

Motivated mostly by works in \cite{DJLW97, YZ17, SZ24, Z24}, we would like to relax conditions \eqref{JLW-cond-1} and \eqref{JLW-cond-2}, since the condition that $h_1$ and $h_2$ are positive somewhere is more natural than the condition that $h_1$ and $h_2$ are positive everywhere. In our former paper \cite{SZhu24+}, we successfully relaxed \eqref{JLW-cond-1} when $h_1$ and $h_2$ can change signs and proved the following theorem.

\begin{thm}[Sun-Zhu \cite{SZhu24+}]\label{thm-sz-1}
Let $(M,g)$ be a compact Riemann surface with the Gaussian curvature $K$. Let $h_1, h_2 \in C^2(M)$ which are positive somewhere and $\rho_2 \in (0, 4\pi)$. Denote $M_1^+ = \{x \in M : h_1(x) > 0\}$. If
\begin{align*}
\Delta \log h_1(x) + (8\pi - \rho_2) - 2K(x) > 0, \quad \forall x \in M_1^+,
\end{align*}
then $J_{4\pi,\rho_2}$ has a minimizer $(u_1,u_2) \in \mathcal{H}$ which satisfies \eqref{eq-jlw-1}.
\end{thm}

In this paper, we shall show that \eqref{JLW-cond-2} in Theorem \ref{thm-critical-0} can also be relaxed to the case where $h_1$ and $h_2$ can change signs. Specifically,

\begin{thm}\label{thm-sz-2}
Let $(M,g)$ be a compact Riemann surface with the Gaussian curvature $K$. Let $h_1, h_2 \in C^2(M)$ which are positive somewhere. Denote
$M_i^+ = \{x \in M : h_i(x) > 0\}$ for $i = 1, 2$. If
\begin{align}\label{sz-cond-2}
\Delta \log h_i(x) + 4\pi - 2K(x) > 0, \quad \forall x \in M_i^+, \quad i = 1, 2,
\end{align}
then $J_{4\pi,4\pi}$ has a minimizer $(u_1,u_2) \in \mathcal{H}$ which satisfies \eqref{eq-jlw-2}.
\end{thm}

At the end of the introduction, we would like to outline the proof of Theorem \ref{thm-sz-2}. For any $\epsilon \in (0, 4\pi)$, we assume that $J_{4\pi-\epsilon, 4\pi-\epsilon}(u_1^\epsilon, u_2^\epsilon) = \inf_{\mathcal{H}} J_{4\pi-\epsilon, 4\pi-\epsilon}$, then $(u_1^\epsilon, u_2^\epsilon)$ satisfies a Toda type system. If $(u_1^\epsilon, u_2^\epsilon)$ converges to some $(u_1, u_2) \in \mathcal{H}$ as $\epsilon \to 0$, then $J_{4\pi, 4\pi}(u_1, u_2) = \inf_{\mathcal{H}} J_{4\pi, 4\pi}$, and we are done. Otherwise, if $(u_1^\epsilon, u_2^\epsilon)$ does not converge in $\mathcal{H}$, we say that $(u_1^\epsilon, u_2^\epsilon)$ blows up. We show that there are three characterizations of the definition of blow-up, one of which is that $\overline{u_1^\epsilon} + \overline{u_2^\epsilon} \to -\infty$ as $\epsilon \to 0$. Here, $\overline{u_i^\epsilon}$ denotes the mean value of $u_i^\epsilon$ on $M$, for $i = 1, 2$. Based on this characterization, we divide the proof into three cases:

\begin{itemize}
    \item Case 1: $\overline{u_1^\epsilon} \to -\infty$ and $\overline{u_2^\epsilon} \geq -C$ as $\epsilon \to 0$.
    \item Case 2: $\overline{u_1^\epsilon} \geq -C$ and $\overline{u_2^\epsilon} \to -\infty$ as $\epsilon \to 0$.
    \item Case 3: $\overline{u_1^\epsilon} \to -\infty$ and $\overline{u_2^\epsilon} \to -\infty$ as $\epsilon \to 0$.
\end{itemize}
Case 1 is similar to the situation where $\rho_1 = 4\pi$ and $\rho_2 \in (0, 4\pi)$, which has been proved by us in \cite{SZhu24+}. Case 2 follows from Case 1. Suppose we are in Case 3. Since $h_1$ and $h_2$ can change signs, we do not have directly the characterization of the blow-up set (Proposition 2.4 in \cite{JLW06}). More effort is needed to understand the blow-up set, which is one of the main contributions in this paper. Since the $L^1$ norm of $e^{u_i^\epsilon}$ is bounded, $e^{u_i^\epsilon} dv_g$ converges to some nonnegative measure $\mu_i$, and $\text{supp} \mu_i \neq \emptyset$, for $i = 1, 2$. By Fatou's lemma, $\text{supp} \mu_i$ is a finite set, for $i = 1, 2$. With the help of the improved Moser-Trudinger inequality for the Toda system, we know that at least one $\text{supp} \mu_i$ is a single point set. By the Pohozaev identity, we derive that $\text{supp} \mu_j$ ($j \neq i$) is also a single point set, which is different from $\text{supp} \mu_i$. Then we can show that $h_1 \mu_1 = \delta_{x_1}$ and $h_2 \mu_2 = \delta_{x_2}$ with $x_1 \neq x_2$. Based on this, we derive a dedicated lower bound for $J_{4\pi, 4\pi}$. Finally, we use the test functions $(\phi_1^\epsilon, \phi_2^\epsilon)$ constructed in \cite{LL05} to show that under condition \eqref{sz-cond-2}, $J_{4\pi, 4\pi}(\phi_1^\epsilon, \phi_2^\epsilon)$ are strictly less than the lower bound derived before. This contradiction tells us that $(u_1^\epsilon, u_2^\epsilon)$ does not blow up, which proves Theorem \ref{thm-sz-2}.

There are some related works which deal with sign-changing potential in the critical case with respect to
Moser-Trudinger type inequalities (cf. \cite{Mart09,YuZ24+}). We believe that our techniques could be used to deal with other nonlinear
existence problems with sign-changing prescribed functions.

The outline of the rest of the paper is following: In Sect. 2, we do some analysis on the minimizing sequence; In Sect. 3, we estimate the lower bound for
$J_{4\pi,4\pi}$ in Case 3; Finally, we complete the proof of Theorem \ref{thm-sz-2} in the last section. Throughout the whole paper,
the constant $C$ is varying from line to line and even in the same line, we do not distinguish
sequence and its subsequences since we only care about the existence result.

\section{Analysis on the minimizing sequence}

In this section, we conduct an analysis on the minimizing sequence.

Given inequality \eqref{ineq-toda}, it is known that for any $\ep \in (0, 4\pi)$, there exists a pair $(u^\ep_1, u^{\ep}_2) \in \mathcal{H}$ such that
\[ J_{4\pi-\ep,4\pi-\ep}(u_1^\ep, u_2^\ep) = \inf_{\mathcal{H}} J_{4\pi-\ep,4\pi-\ep}. \]
Direct calculations reveal the following equations on $M$:
\begin{align}\label{eq-uep}
\begin{cases}
-\Delta u^\ep_1 = (8\pi - 2\ep)(h_1 e^{u_1^\ep} - 1) - (4\pi - \ep)(h_2 e^{u^\ep_2} - 1), \\
-\Delta u^\ep_2 = (8\pi - 2\ep)(h_2 e^{u^\ep_2} - 1) - (4\pi - \ep)(h_1 e^{u_1^\ep} - 1).
\end{cases}
\end{align}

Let $\overline{u^\ep_i} = \int_M u^\ep_i$ and $m^\ep_i = \max_{M} u^\ep_i = u_i^\ep(x_i^\ep)$ for some $x_i^{\ep} \in M$. Assume $x_i^\ep \to p_i$ as $\ep \to 0$. The following three lemmas can be derived similarly to those in \cite[section 2]{SZhu24+}.

\begin{lem}\label{lem-bound}
There exist two positive constants $C_1$ and $C_2$ such that
\[ C_1 \leq \int_M e^{u_i^{\ep}} \leq C_2, \quad i = 1, 2. \]
\end{lem}

\begin{lem}\label{lem-Ls}
For any $s \in (1, 2)$, $\|\nabla u_i^\ep\|_{L^s(M)} \leq C$ for $i = 1, 2$.
\end{lem}

\begin{lem}\label{lem-char}
The following statements are equivalent:
\begin{enumerate}[label=(\roman*)]
    \item $m_1^{\ep} + m_2^{\ep} \to +\infty$ as $\ep \to 0$,
    \item $\int_M (|\nabla u^\ep_1|^2 + \nabla u_1^\ep \cdot \nabla u_2^\ep + |\nabla u^\ep_2|^2) \to +\infty$ as $\ep \to 0$,
    \item $\overline{u^\ep_1} + \overline{u^\ep_2} \to -\infty$ as $\ep \to 0$.
\end{enumerate}
\end{lem}

\begin{definition}[Blow-Up]
We say $(u_1^\ep, u_2^\ep)$ blows up if any one of the conditions in Lemma \ref{lem-char} holds.
\end{definition}

If $(u_1^\ep,u_2^\ep)$ does not blow up, then by Lemma \ref{lem-char}, one can show that $(u_1^\ep,u_2^\ep)$ converges to some $(u_1,u_2)$ in $\mathcal{H}$ which minimizes $J_{4\pi,4\pi}$. The proof of Theorem \ref{thm-sz-2} terminates in this case. Therefore, without loss of generality we may assume $(u_1^\ep,u_2^\ep)$ blows up in the rest of this paper.

By Lemma \ref{lem-char} (iii), we divide the proof into the following three cases:

\textbf{Case 1} $\overline{u_1^\ep}\to-\infty$, $\overline{u_2^\ep}\geq-C$ as $\ep\to0$;

\textbf{Case 2} $\overline{u_1^\ep}\geq-C$, $\overline{u_2^\ep}\to-\infty$ as $\ep\to0$;

\textbf{Case 3} $\overline{u_1^\ep}\to-\infty$, $\overline{u_2^\ep}\to-\infty$ as $\ep\to0$.

Suppose we are in Case 1, by checking the proofs in \cite{SZhu24+} carefully, we find that $\rho_2<4\pi$ is used to show $\overline{u_2^\ep}\geq-C$ which happens to be the situation in Case 1. And at any other places $\rho_2<4\pi$ can be replaced by $\rho_2=4\pi$. By Theorem \ref{thm-sz-1}, if
$$\Delta \log h_1(x)+4\pi-2K(x)>0~~\text{for}~x\in M_1^+,$$
where $M_1^+=\{x\in M:~h_1(x)>0\}$, $J_{4\pi,4\pi}$ has a minimizer $(u_1,u_2)\in\mathcal{H}$ which satisfies \eqref{eq-jlw-2}.

Suppose we are in Case 2, similar as Case 1, we know that, if
$$\Delta \log h_2(x)+4\pi-2K(x)>0~~\text{for}~x\in M_2^+,$$
where $M_2^+=\{x\in M:~h_2(x)>0\}$, then $J_{4\pi,4\pi}$ has a minimizer $(u_1,u_2)\in\mathcal{H}$ which satisfies \eqref{eq-jlw-2}.

Suppose we are in Case 3, by Lemma \ref{lem-Ls}, there exist $G_i$, $i=1,2$ such that $u_i^\ep-\overline{u_i^\ep}\rightharpoonup G_i$ weakly in $W^{1,s}(M)$ for any $1<s<2$ as $\ep\to0$. Since $(e^{u_i^\ep})$ is bounded in $L^1(M)$ we may extract a subsequence (still denoted $e^{u_i^\ep}$) such that $e^{u_i^\ep}$ converges in the sense of measures on $M$ to some nonnegative bounded measure $\mu_i$ for $i=1,2$. We set
\begin{align*}
\gamma_1=8\pi h_1\mu_1-4\pi h_2\mu_2,~~\gamma_2=8\pi h_2\mu_2-4\pi h_1\mu_1
\end{align*}
and
\begin{align*}
S_i=\{x\in M: |\gamma_i(\{x\})|\geq4\pi\},~i=1,2.
\end{align*}
Let $S=S_1\cup S_2$. By Theorem 1 in \cite{BM}, we have
\begin{lem}
For any $\Omega\subset\subset M\setminus S$, there holds
\begin{align}\label{outbound}
u_i^\ep-\overline{u_i^\ep}~\text{is~uniformly~bounded~in~} \Omega,~i=1,2.
\end{align}
\end{lem}
\begin{proof}
$\forall x\in M\setminus S$, we have $B_\delta(x)\subset\subset M\setminus S$ for sufficiently small $\delta>0$. Consider the equation
\begin{align*}
\begin{cases}
-\Delta w_{1}^\ep = (8\pi-2\ep) h_1e^{u_1^\ep}-(4\pi-\ep)h_{2}e^{u_{2}^\ep}:=f_\ep &\text{in}~~B_\delta(x),\\
w_{1}^\ep = 0 &\text{on}~~\partial B_\delta(x).
\end{cases}
\end{align*}
Since $|\gamma_1(\{x\})|<4\pi$, we have $\|f_\ep\|_{L^1(B_{\delta}(x))}<4\pi$ for sufficiently small $\ep>0$ and $\delta>0$. Fix such a $\delta$.
Define $w_2^\ep=u_1^\ep-\overline{u_1^\ep}-w_1^\ep$, then $-\Delta w_2^\ep = -(4\pi-\ep)$ in $B_\delta(x)$.
By Theorem 4.1 in \cite{HL} and Lemma \ref{lem-Ls}, we have
\begin{align*}
\sup_{B_{\delta/2}(x)}w_2^\ep \leq& C\left(\|w_2^\ep\|_{L^1(B_\delta(x))}+C\right)\\
\leq& C\left(\|u_1^\ep-\overline{u_1^\ep}\|_{L^1(M)}+\|v_1^\ep\|_{L^1(B_{\delta}(x))}+C\right)\\
\leq& C\left(\|\nabla u_1^\ep\|_{L^s(M)}+\|w_1^\ep\|_{L^1(B_{\delta}(x))}+C\right)\\
\leq& C\left(\|w_1^\ep\|_{L^1(B_{\delta}(x))}+C\right).
\end{align*}
It follows from Theorem 1 in \cite{BM} that $e^{s_1|w_1^\ep|}$ is bounded in $B_{\delta}(x)$ for some $s_1>1$, which yields that
$$\|w_1^\ep\|_{L^1(B_\delta(x))}\leq C.$$
So we have
$$\sup_{B_{\delta/2}(x)}w_2^\ep\leq C.$$
Then
\begin{align*}
\int_{B_{\delta/2}(x)}e^{s_1u_1^\ep}=&\int_{B_{\delta/2}(x)}e^{s_1\overline{u_1^\ep}}e^{s_1w_2^\ep}e^{s_1w_1^\ep}\\
\leq& C\int_{B_{\delta/2}(x)}e^{s_1|w_1^\ep|}\\
\leq& C.
\end{align*}
Similarly, we have $\int_{B_{\delta/2}(x)}e^{s_2u_2^\ep}\leq C$ for some $s_2>1$ and sufficiently small $\delta>0$. Then \eqref{outbound} follows from the standard elliptic estimates and we finish the proof.
\end{proof}

Since $(u_1^\ep, u_2^\ep)$ blows up, we know $S$ is not empty. Otherwise, by using a finite covering argument and inequality \eqref{outbound}, we would have $\|u_i^\ep - \overline{u_i^\ep}\|_{L^\infty(M)} \leq C$, which implies $m^\ep_i \leq C$. This contradicts Lemma \ref{lem-char} (i). By the definition of $S$, for any $x \in S$,
$$\mu_1(\{x\}) \geq \frac{1}{4\max_M |h_1|} \quad \text{or} \quad \mu_2(\{x\}) \geq \frac{1}{4\max_M |h_2|}.$$
In view of $\mu_1$ and $\mu_2$ are bounded, $S$ is a finite set. We denote $S=\{x_l\}_{l=1}^L$.
It follows from
\eqref{outbound} and Fatou's lemma that
\begin{lem}\label{lem-mu_i} We have
\begin{align}\label{mu_i=}
\mu_i=\sum_{l=1}^L\mu_i(\{x_l\})\delta_{x_l},~~i=1,2,
\end{align}
where $\delta_{x}$ is the Dirac distribution.
\end{lem}
\begin{proof}
For any closed set $V\subset M$, we need to show
\begin{align}\label{mu_i-1}
\mu_i(V) = \sum_{l=1}^L\mu_i(V\cap\{x_l\}), ~i=1,2.
\end{align}
In fact, we have $B_r(x_l)\cap B_r(x_m)=\emptyset$ for sufficiently small $r$ and $l\neq m\in\{1,\cdots,L\}$.
Then
\begin{align}\label{mu_i-2}
\int_V e^{u^\ep_i} =& \int_{V\setminus \bigcup_{l=1}^L B_r(x_l)} e^{u^\ep_i} + \int_{V\cap\left(\bigcup_{l=1}^L B_r(x_l)\right)} e^{u^\ep_i}\nonumber\\
=& \int_{V\setminus \bigcup_{l=1}^L B_r(x_l)} e^{u^\ep_i-\overline{u_i^\ep}} e^{\overline{u_i^\ep}}+ \sum_{l=1}^L\int_{V\cap B_r(x_l)} e^{u^\ep_i}.
\end{align}
Since $\overline{u_i^\ep}\to-\infty$ and \eqref{outbound}, it follows from Fatou's lemma that
$$\liminf_{\ep\to0}\int_{V\setminus \bigcup_{l=1}^L B_r(x_l)} e^{u^\ep_i-\overline{u_i^\ep}} e^{\overline{u_i^\ep}}=0.$$
Letting $\ep\to0$ in both sides of \eqref{mu_i-2} first and then $r\to0$, we obtain \eqref{mu_i-1} and finish the proof.
\end{proof}

It follows from Lemma \ref{lem-bound} and \eqref{mu_i=} that ${\rm{supp}}\mu_i\neq\emptyset$, $i=1,2$. If there are at least two points in each ${\rm{supp}}\mu_i$, then by the improved Moser-Trudinger inequality for Toda system (cf. \cite[Proposition 2.5]{MR13}), for any $\ep'>0$, there exists some $C=C(\ep')>0$ such that
\begin{align*}
\log\int_Me^{u_1^\ep}+\log\int_Me^{u_2^\ep}&\leq\frac{1+\ep'}{24\pi}\int_M(|\nabla u_1^\ep|^2+\nabla u_1^\ep\nabla u_2^\ep+|\nabla u_2^\ep|^2)+\overline{u_1^\ep}+\overline{u_2^\ep}+C.
\end{align*}
By choosing $\ep'=1/3$ and using Lemma \ref{lem-bound}, we have
\begin{align*}
\frac{1}{3}\int_M(|\nabla u_1^\ep|^2+\nabla u_1^\ep\nabla u_2^\ep+|\nabla u_2^\ep|^2)&\geq-6\pi(\overline{u_1^\ep}+\overline{u_2^\ep})-C.
\end{align*}
This, combining with the fact that $J_{4\pi-\ep,4\pi-\ep}(u_1^\ep,u_2^\ep)$ is bounded, shows that
\begin{align*}
\overline{u_1^\ep}+\overline{u_2^\ep}&\geq-C,
\end{align*}
which contradicts the assumption that $(u_1^\ep,u_2^\ep)$ blows up. Hence, either ${\rm{supp}}\mu_1$ or ${\rm{supp}}\mu_2$ has only one point. Without loss of generality, we assume that ${\rm{supp}}\mu_1$ has only one point and ${\rm{supp}}\mu_1=\{x_1\}$ since ${\rm{supp}}\mu_1\subset S$. By noticing that $\int_Mh_1e^{u_1^\ep}=1$, we have
\begin{align}\label{h_1mu_1}
h_1\mu_1=\delta_{x_1}.
\end{align}

The following result which is based on Pohozaev identities is very important in the understanding of blow-up set.
\begin{lem}\label{lem-poho}
Denote by $h_i\mu_i=\sigma_i$ for $i=1,2$, we have
\begin{align}\label{pohozaev8}
\sigma_1^2(\{x_l\})+\sigma_2^2(\{x_l\})-\sigma_1(\{x_l\})\sigma_2(\{x_l\})=\sigma_1(\{x_l\})+\sigma_2(\{x_l\}),
\end{align}
where $l=1,2,\cdots,L$.
\end{lem}
\begin{proof}
Using \eqref{outbound}, \eqref{mu_i=} and \eqref{h_1mu_1}, we have $G_1$ and $G_2$ satisfy the following equation
\begin{align}\label{eq-green-1}
\begin{cases}
-\Delta G_1 = 8\pi(\delta_{x_1}-1)-4\pi(h_2\sum_{l=1}^L\mu_2(\{x_l\})\delta_{x_l}-1),\\
-\Delta G_2 = 8\pi(h_2\sum_{l=1}^L\mu_2(\{x_l\})\delta_{x_l}-1)-4\pi(\delta_{x_1}-1),\\
\int_MG_1=\int_MG_2=0.
\end{cases}
\end{align}
It follows from standard elliptic estimates that
\begin{align}\label{converge-out}
u_i^\ep-\overline{u_i^\ep}\to G_i~~\text{in}~~C^2_{\text{loc}}(M\setminus S),~i=1,2.
\end{align}
Let $(B_{\delta}(x_l);(x^1,x^2))$ be an isothermal coordinate system around $x_l$ and we assume the metric to be
$$g|_\Omega = e^{\phi}((dx^1)^2+(dx^2)^2)$$
with $\phi(0)=0$ and $\nabla_{\mathbb{R}^2}\phi(0)=0$. Here $\nabla_{\mathbb{R}^2}=(\frac{\p}{\partial x^1},\frac{\p}{\partial x^2})$.
It is well known that
\begin{align}\label{pohozaev0}
G_i = -\frac{\gamma_i(\{x_l\})}{2\pi}\log r + \psi_i,~~i=1,2,
\end{align}
where $r=\sqrt{(x^1)^2+(x^2)^2}$ and $\psi_i$ is a smooth function near $x_l$. In this coordinate system, \eqref{eq-uep} can be reduced to
\begin{align}\label{pohozaev1}
\begin{cases}
-\Delta_{\mathbb{R}^2}u_1^\ep = (8\pi-2\ep)e^{\phi}(h_1e^{u_1^\ep}-1)-(4\pi-\ep)e^{\phi}(h_2e^{u_2^\ep}-1),\\
-\Delta_{\mathbb{R}^2}u_2^\ep = (8\pi-2\ep)e^{\phi}(h_2e^{u_2^\ep}-1)-(4\pi-\ep)e^{\phi}(h_1e^{u_1^\ep}-1)
\end{cases}
\end{align}
for $|x|\leq\delta$, where $\Delta_{\mathbb{R}^2}=\frac{\p^2}{\p(x^1)^2}+\frac{\p^2}{\p(x^2)^2}$ is the Laplacian in $\mathbb{R}^2$.
We set
\begin{align*}
\hat{u}_i^\ep(x) = u_i^\ep(x)-(4\pi-\ep)\zeta(x),
\end{align*}
where $\zeta(x)$ satisfies
\begin{align*}
\begin{cases}
\Delta_{\mathbb{R}^2}\zeta = e^{\phi(x)}~~\text{for}~|x|\leq\delta,\\
\zeta(0)=0~~\text{and}~~\nabla_{\mathbb{R}^2}\zeta(0)=0.
\end{cases}
\end{align*}
It is clear that $\zeta(x)=O(|x|^2)$ for $|x|\leq\delta$. By \eqref{pohozaev1} we know $\hat{u}_i^\ep$ satisfies
\begin{align}\label{pohozaev2}
\begin{cases}
-\Delta_{\mathbb{R}^2}\hat{u}_1^\ep = (8\pi-2\ep)\hat{h}_1e^{\hat{u}_1^\ep}-(4\pi-\ep)\hat{h}_2e^{\hat{u}_2^\ep},\\
-\Delta_{\mathbb{R}^2}\hat{u}_2^\ep = (8\pi-2\ep)\hat{h}_2e^{\hat{u}_2^\ep}-(4\pi-\ep)\hat{h}_1e^{\hat{u}_1^\ep}
\end{cases}
\end{align}
for $|x|\leq\delta$, where
\begin{align}\label{pohozaev3}
\hat{h}_i(x) = e^{\phi(x)}h_i(x)e^{(4\pi-\ep)\zeta(x)},~~i=1,2.
\end{align}
It follows from the choice of $\phi(x)$ and \eqref{pohozaev3} that
\begin{align}\label{pohozaev4}
\hat{h}_i(0) = h_i(x_l)~~\text{and}~~\nabla_{\mathbb{R}^2} \hat{h}_i(0) = \nabla h_i(x_l).
\end{align}
From equation \eqref{pohozaev2} we have the Pohozaev identities as follows:
\begin{align}\label{pohozaev5-1}
&-\delta\int_{\p\mathbb{B}_{\delta}(0)}\left(\left(\frac{\p\hat{u}_1^\ep}{\p r}\right)^2-\frac{1}{2}|\nabla_{\mathbb{R}^2}\hat{u}_1^\ep|^2\right)ds\nonumber\\
=& (8\pi-2\ep)\delta\int_{\p\mathbb{B}_{\delta}(0)}\hat{h}_1e^{\hat{u}_1^\ep}ds
  - (8\pi-2\ep)\int_{\mathbb{B}_{\delta}(0)}(2\hat{h}_1e^{\hat{u}_1^\ep}+x\cdot\nabla_{\mathbb{R}^2}\hat{h}_1e^{\hat{u}_1^\ep})dx\nonumber\\
  &- (4\pi-\ep)\int_{\mathbb{B}_{\delta}(0)}x\cdot\nabla_{\mathbb{R}^2}\hat{u}_1^\ep \hat{h}_2e^{\hat{u}_2^\ep}dx,
\end{align}
and
\begin{align}\label{pohozaev5-2}
&-\delta\int_{\p\mathbb{B}_{\delta}(0)}\frac{\p\hat{u}_1^\ep}{\p r}\frac{\p\hat{u}_2^\ep}{\p r}ds
+\int_{\mathbb{B}_{\delta}(0)}\left(\nabla_{\mathbb{R}^2}\hat{u}_1^\ep\nabla_{\mathbb{R}^2}\hat{u}_2^\ep
+\sum_{j=1}^2x\cdot\left(\nabla_{\mathbb{R}^2}\frac{\p \hat{u}_2^\ep}{\p x^j}\right)\frac{\p \hat{u}_1^\ep}{\p x^j}\right)dx\nonumber\\
=& -(4\pi-\ep)\delta\int_{\p\mathbb{B}_{\delta}(0)}\hat{h}_2e^{\hat{u}_2^\ep}ds
   +(8\pi-2\ep)\int_{\mathbb{B}_{\delta}(0)}\left(\hat{h}_2e^{\hat{u}_2^\ep}+x\cdot\nabla_{\mathbb{R}^2}\hat{u}_2^\ep \hat{h}_1e^{\hat{u}_1^\ep}\right)dx\nonumber\\
  &+(4\pi-\ep)\int_{\mathbb{B}_{\delta}(0)}x\cdot\nabla_{\mathbb{R}^2}\hat{h}_2e^{\hat{u}_2^\ep}dx,
\end{align}
and
\begin{align}\label{pohozaev5-3}
&-\delta\int_{\p\mathbb{B}_{\delta}(0)}\frac{\p\hat{u}_1^\ep}{\p r}\frac{\p\hat{u}_2^\ep}{\p r}ds
+\int_{\mathbb{B}_{\delta}(0)}\left(\nabla_{\mathbb{R}^2}\hat{u}_1^\ep\nabla_{\mathbb{R}^2}\hat{u}_2^\ep
+\sum_{j=1}^2x\cdot\left(\nabla_{\mathbb{R}^2}\frac{\p \hat{u}_1^\ep}{\p x^j}\right)\frac{\p \hat{u}_2^\ep}{\p x^j}\right)dx\nonumber\\
=& -(4\pi-\ep)\delta\int_{\p\mathbb{B}_{\delta}(0)}\hat{h}_1e^{\hat{u}_1^\ep}ds
   +(8\pi-2\ep)\int_{\mathbb{B}_{\delta}(0)}\left(\hat{h}_1e^{\hat{u}_1^\ep}+x\cdot\nabla_{\mathbb{R}^2}\hat{u}_1^\ep \hat{h}_2e^{\hat{u}_2^\ep}\right)dx\nonumber\\
  &+(4\pi-\ep)\int_{\mathbb{B}_{\delta}(0)}x\cdot\nabla_{\mathbb{R}^2}\hat{h}_1e^{\hat{u}_1^\ep}dx,
\end{align}
and
\begin{align}\label{pohozaev5-4}
&-\delta\int_{\p\mathbb{B}_{\delta}(0)}\left(\left(\frac{\p\hat{u}_2^\ep}{\p r}\right)^2-\frac{1}{2}|\nabla_{\mathbb{R}^2}\hat{u}_2^\ep|^2\right)ds\nonumber\\
=& (8\pi-2\ep)\delta\int_{\p\mathbb{B}_{\delta}(0)}\hat{h}_2e^{\hat{u}_2^\ep}ds
  - (8\pi-2\ep)\int_{\mathbb{B}_{\delta}(0)}(2\hat{h}_2e^{\hat{u}_2^\ep}+x\cdot\nabla_{\mathbb{R}^2}\hat{h}_2e^{\hat{u}_2^\ep})dx\nonumber\\
  &- (4\pi-\ep)\int_{\mathbb{B}_{\delta}(0)}x\cdot\nabla_{\mathbb{R}^2}\hat{u}_2^\ep \hat{h}_1e^{\hat{u}_1^\ep}dx.
\end{align}
Two times both sides of \eqref{pohozaev5-1} and \eqref{pohozaev5-4} and then plus each sides of them with \eqref{pohozaev5-2} and \eqref{pohozaev5-3}, we have
\begin{align}\label{pohozaev6}
&-2\delta\int_{\p\mathbb{B}_{\delta}(0)}\left(\left(\frac{\p\hat{u}_1^\ep}{\p r}\right)^2+\left(\frac{\p\hat{u}_2^\ep}{\p r}\right)^2+\frac{\p\hat{u}_1^\ep}{\p r}\frac{\p\hat{u}_2^\ep}{\p r}\right)ds\nonumber\\
&+\delta\int_{\p\mathbb{B}_{\delta}(0)}\left(|\nabla_{\mathbb{R}^2}\hat{u}_1^\ep|^2+|\nabla_{\mathbb{R}^2}\hat{u}_2^\ep|^2
+\nabla_{\mathbb{R}^2}\hat{u}_1^\ep\nabla_{\mathbb{R}^2}\hat{u}_2^\ep\right)ds\nonumber\\
=&3(4\pi-\ep)\delta\int_{\p\mathbb{B}_{\delta}(0)}\left(\hat{h}_1e^{\hat{u}_1^\ep}+\hat{h}_2e^{\hat{u}_2^\ep}\right)ds
  -6(4\pi-\ep)\int_{\mathbb{B}_{\delta}(0)}\left(\hat{h}_1e^{\hat{u}_1^\ep}+\hat{h}_2e^{\hat{u}_2^\ep}\right)dx\nonumber\\
 &-3(4\pi-\ep)\int_{\mathbb{B}_{\delta}(0)}\left(x\cdot\nabla_{\mathbb{R}^2}\hat{h}_1e^{\hat{u}_1^\ep}+x\cdot\nabla_{\mathbb{R}^2}\hat{h}_2e^{\hat{u}_2^\ep}\right)dx.
\end{align}
Letting $\ep\to0$ first and then $\delta\to0$ in \eqref{pohozaev6}, by using \eqref{converge-out}, \eqref{pohozaev0}, \eqref{pohozaev3} and \eqref{pohozaev4} we conclude
\begin{align}\label{pohozaev7}
 &-2\pi\left[\left(\frac{\gamma_1(\{x_l\})}{2\pi}\right)^2+\left(\frac{\gamma_2(\{x_l\})}{2\pi}\right)^2+\frac{\gamma_1(\{x_l\})}{2\pi}\frac{\gamma_2(\{x_l\})}{2\pi}\right]\nonumber\\
=&-24\pi\left[h_1(x_l)\mu_1(\{x_l\})+h_2(x_l)\mu_2(\{x_l\})\right].
\end{align}
Recalling that $h_i\mu_i=\sigma_i$ for $i=1,2$, then \eqref{pohozaev7} reduces to \eqref{pohozaev8}, this ends the proof.
\end{proof}

Now we show by Lemma \ref{lem-poho} that ${\rm{supp}}\mu_2$ also has one point which is different with $x_1$.

We know from \eqref{h_1mu_1} that $\sigma_1(\{x_1\})=1$ and $\sigma_1(\{x_l\})=0$ for any $l\geq2$, taking this fact into \eqref{pohozaev8} we obtain that
\begin{align*}
&\sigma_2(\{x_1\})=0~~\text{or}~~\sigma_2(\{x_1\})=2;\\
&\sigma_2(\{x_l\})=0~~\text{or}~~\sigma_2(\{x_l\})=1,~~\forall l\geq2.
\end{align*}
Combining it with $\sigma_2(M)=\int_{M}h_2e^{u_2^\ep}=1$, we have
\begin{align*}
\sigma_2(\{x_m\})=1~~\text{for~some}~m\geq2~~\text{and}~\sigma_2(\{x_l\})=0~~\forall l\in\{1,\cdots,L\}\setminus\{m\}.
\end{align*}
Without loss of generality, we assume $m=2$. Then we have
\begin{align}\label{h_2mu_2}
h_2\mu_2=\sigma_2=\delta_{x_2}.
\end{align}
We would like to collect \eqref{h_1mu_1} and \eqref{h_2mu_2} as the following lemma.
\begin{lem}\label{mu_1mu_2}
It holds that $h_1\mu_1=\delta_{x_1}$ and $h_2\mu_2=\delta_{x_2}$ with $x_1\neq x_2$.
\end{lem}
To do blow-up analysis near $x_i$ for $i=1,2$, one still needs the upper bound of $u^\ep_j$ for $j\in\{1,2\}\setminus\{i\}$ near $x_i$. In fact, we have
\begin{lem}\label{upperbound}
Suppose $r$ is a positive number which is less than $\text{dist}(x_1,x_2)/2$ and makes $h_i>0$ in $B_r(x_i)$ for $i=1,2$, there holds
\begin{align*}
\sup_{B_{r/4}(x_i)}\left(u_j^\ep-\overline{u_j^\ep}\right) \leq C,\quad i,j\in\{1,2\}~\text{and}~i\neq j.
\end{align*}
\end{lem}
\begin{proof} For $i=1$, we consider the solution of
\begin{align*}
\begin{cases}
-\Delta v_{1}^\ep = (8\pi-2\ep)h_2e^{u_2^\ep} &\text{in}~B_r(x_1),\\
v_{1}^\ep=0 &\text{on}~\p B_r(x_1).
\end{cases}
\end{align*}
Denote by $v_2^\ep=u_2^\ep-\overline{u_2^\ep}-v_1^\ep$, then
\begin{align*}
-\Delta v_2^\ep = -(4\pi-\ep)-(4\pi-\ep)h_1e^{u_1^\ep}\leq-(4\pi-\ep)~~\text{in}~B_r(x_1)
\end{align*}
since $h_1>0$ in $B_r(x_1)$. By Theorem 8.17 in \cite{GT} (or Theorem 4.1 in \cite{HL}) and Lemma \ref{lem-Ls}, we have
\begin{align*}
\sup_{B_{r/2}(x_1)}v_2^\ep \leq& C\left(\|(v_2^\ep)^+\|_{L^s(B_r(x_1))}+C\right)\\
\leq& C\left(\|u_2^\ep-\overline{u_2^\ep}\|_{L^s(M)}+\|v_1^\ep\|_{L^s(B_r(x_1))}+C\right)\\
\leq& C\left(\|\nabla u_2^\ep\|_{L^s(M)}+\|v_1^\ep\|_{L^s(B_r(x_1))}+C\right)\\
\leq& C\left(\|v_1^\ep\|_{L^s(B_r(x_1))}+C\right).
\end{align*}
Since $\int_{B_r(x_1)}|h_2|e^{u_2^\ep}\to0$ as $\ep\to0$, it follows from Theorem 1 in \cite{BM} that $\int_{B_r(x_1)}e^{t|v_1^\ep|}\leq C$ for some $t>1$, which yields that
$$\|v_1^\ep\|_{L^s(B_r(x_1))}\leq C.$$
Then we have
$$\sup_{B_{r/2}(x_1)}v_2^\ep\leq C.$$
Note that
\begin{align*}
\int_{B_{r/2}(x_1)}e^{tu_2^\ep}=&\int_{B_{r/2}(x_1)}e^{t\overline{u_2^\ep}}e^{tv_2^\ep}e^{tv_1^\ep}\\
\leq& C\int_{B_{r/2}(x_1)}e^{t|v_1^\ep|}\\
\leq& C.
\end{align*}
By the standard elliptic estimates, we have
\begin{align*}
\|v_1^\ep\|_{L^{\infty}(B_{r/4}(x_1))} \leq C.
\end{align*}
Therefore, we obtain that
\begin{align*}
u_2^\ep-\overline{u_2^\ep}\leq C~~\text{in}~~B_{r/4}(x_1).
\end{align*}
Similarly, we can prove
\begin{align*}
u_1^\ep-\overline{u_1^\ep}\leq C~~\text{in}~~B_{r/4}(x_2).
\end{align*}
This finishes the proof.
\end{proof}

Recalling that $\overline{u_i^\ep}\to-\infty$ and $\max_{M}u_i^\ep(x)=u_i^\ep(x_i^\ep)$, $i=1,2$, it follows from \eqref{converge-out}, Lemmas \ref{mu_1mu_2} and \ref{upperbound} that
\begin{align*}
x_i^\ep\to x_i~~\text{as}~~\ep\to0,~i=1,2.
\end{align*}
Let $(\Omega_i;(x^1,x^2))$ be an isothermal coordinate system around $x_i$ and we assume the metric to be
$$g|_{\Omega_i} = e^{\phi_i}((dx^1)^2+(dx^2)^2),~~\phi_i(0)=0.$$
Similar as Case 1 in \cite{LL05} and Lemma 2.5 in \cite{DJLW97}, we have
\begin{align*}
u^\ep_i(x^\ep_i+r_i^{\ep}x)-m_i^\ep\to -2\log(1+\pi h_i(x_i)|x|^2),~~i=1,2,
\end{align*}
where $m_i^\ep=\max_Mu_i^\ep$ and $r_i^\ep=e^{-m_i^\ep/2}$.

By taking \eqref{h_2mu_2} into \eqref{eq-green-1}, we have
\begin{align*}
\begin{cases}
-\Delta G_1 = 8\pi(\delta_{x_1}-1)-4\pi(\delta_{x_2}-1),\\
-\Delta G_2 = 8\pi(\delta_{x_2}-1)-4\pi(\delta_{x_1}-1),\\
\int_MG_1=\int_MG_2=0.
\end{cases}
\end{align*}
Recalling that for any $s\in(1,2)$, for $i=1,2$, we have
$u^\ep_i-\overline{u^\ep_i}\to G_i$ weakly in $W^{1,s}(M)$ and strongly in $C^2_{\text{loc}}(M\setminus\{x_1,x_2\})$ as $\ep\to0$.

It was proved by Li-Li in \cite[page 708]{LL05} that, in $\Omega_1$,
\begin{align*}
G_1(x,x_1)=-4\log r+A_1(x_1)+f_1,~~G_2(x,x_1)=2\log r+A_2(x_1)+g_1,
\end{align*}
where $r^2=x_1^2+x_2^2$, $A_i(x_1)$ $(i=1,2)$ are constants and $f_1,g_1$ are two smooth functions which are zero at $x_1$.
In $\Omega_2$,
\begin{align*}
G_1(x,x_2)=2\log r+A_1(x_2)+f_2,~~G_2(x,x_2)=-4\log r+A_2(x_2)+g_2,
\end{align*}
where $A_i(x_2)$ $(i=1,2)$ are constants and $f_2,g_2$ are two smooth functions which are zero at $x_2$.

\section{ The lower bound for \texorpdfstring{$J_{4\pi,4\pi}$}{J} in Case 3}
In this section, we shall derive an explicit lower bound of $J_{4\pi,4\pi}$ under the assumptions $(u_1^\ep,u_2^\ep)$ blows up and Case 3 happens.

Following closely the calculations in \cite[Section 3]{LL05}, we have
\begin{align*}
J_{4\pi-\ep,4\pi-\ep}(u_1^\ep,u_2^\ep)\geq&-4\pi-4\pi\log(\pi h_1(x_1))-2\pi A_1(x_1)\nonumber\\
 &-4\pi-4\pi\log(\pi h_2(x_2))-2\pi A_2(x_2)+o_{\ep}(1)+o_{L}(1)+o_{\delta}(1).
\end{align*}
By letting $\ep\to0$ first, then $L\to+\infty$ and then $\delta\to0$, we obtain finally that
\begin{align}\label{lower-bound}
\inf_{\mathcal{H}}J_{4\pi,4\pi}\geq&-4\pi-4\pi\log(\pi h_1(x_1))-2\pi A_1(x_1)\nonumber\\
 &-4\pi-4\pi\log(\pi h_2(x_2))-2\pi A_2(x_2)\nonumber\\
 \geq&-8\pi-8\pi\log\pi-2\pi\max_{x\in M_1^+}\left(2\log h_1(x)+A_1(x)\right)\nonumber\\
 &-2\pi\max_{x\in M_2^+}\left(2\log h_2(x)+A_2(x)\right).
\end{align}

\section{Completion of the proof of Theorem \ref{thm-sz-2}}
In this section, we shall use the test functions constructed in \cite{LL05} to finish the proof of our main theorem.

Let $\phi_1^\ep$ and $\phi_2^\ep$ be defined as \cite[Section 5]{LL05}. Suppose that
$$2\log h_i(p_i)+A_i(p_i)=\max_{x\in M_i^+}(2\log h_i(x)+A_i(x))~\text{for}~i=1,2.$$
Following directly the calculations in \cite[Section 5]{LL05} and \cite[Section 4]{SZhu24+}, we obtain
\begin{align*}
J_{4\pi,4\pi}(\phi_1^\ep,\phi_2^\ep)
\leq&-8\pi-8\pi\log\pi-4\pi\log h_1(p_1)-2\pi A_1(p_1)-4\pi\log h_2(p_2)-2\pi A_2(p_2)\nonumber\\
 &-\left[\Delta\log h_1(p_1)+4\pi-2K(p_1)\right]\ep^2(-\log\ep^2)\nonumber\\
 &-\left[\Delta\log h_2(p_2)+4\pi-2K(p_2)\right]\ep^2(-\log\ep^2)\nonumber\\
 &+o(\ep^2(-\log\ep^2)).
\end{align*}
Then under the condition \eqref{sz-cond-2}, we have for sufficiently small $\ep$ that
\begin{align*}
J_{4\pi,4\pi}(\phi_1^\ep,\phi_2^\ep)
<-8\pi-8\pi\log\pi-4\pi\log h_1(p_1)-2\pi A_1(p_1)-4\pi\log h_2(p_2)-2\pi A_2(p_2).
\end{align*}

It is easy to check that $\int_{M}h_ie^{\phi_i^\ep}>0$ for $i=1,2$, we define
\begin{align*}
\widetilde{\phi_i^\ep}=\phi_i^\ep-\log\int_{M}h_ie^{\phi_i^\ep},~~i=1,2.
\end{align*}
Then $(\widetilde{\phi_1^\ep},\widetilde{\phi_2^\ep})\in\mathcal{H}$. Since $J_{4\pi,4\pi}(u_1+c_1,u_2+c_2)=J_{4\pi,4\pi}(u_1,u_2)$ for any $c_1,c_2\in\mathbb{R}$, we have for sufficiently small $\ep$ that
\begin{align}\label{up-bound-last}
\inf_{\mathcal{H}}J_{4\pi,4\pi}\leq& J_{4\pi,4\pi}(\widetilde{\phi_1^\ep},\widetilde{\phi_2^\ep})=J_{4\pi,4\pi}(\phi_1^\ep,\phi_2^\ep)\nonumber\\
<&-8\pi-8\pi\log\pi-2\pi\max_{x\in M_+}\left(2\log h_1(x)+A_1(x)\right)\nonumber\\
 &-2\pi\max_{x\in M_2^+}\left(2\log h_2(x)+A_2(x)\right).
\end{align}
Combining \eqref{lower-bound} and \eqref{up-bound-last}, one knows that $(u_1^\ep,u_2^\ep)$ does not blow up. So $(u_1^\ep,u_2^\ep)$ converges to some $(u_1,u_2)$
which minimizes $J_{4\pi,4\pi}$ in $\mathcal{H}$ and solves \eqref{eq-jlw-2}. The smooth of $u_1$ and $u_2$ follows from the standard elliptic estimates. Finally,
we complete the proof of Theorem \ref{thm-sz-2}.   $\hfill{\square}$

\vspace{1cm}

\noindent\textbf{Data Availability} Data sharing is not applicable to this article as obviously no datasets were generated or
analyzed during the current study.

\noindent\textbf{Conflict of interest} The authors have no Conflict of interest to declare that are relevant to the content of this
article.

\end{document}